\documentclass{svmult}

\renewcommand{\email}[1]{\emailname: #1} 

\usepackage{mathptmx}       
\usepackage{helvet}         
\usepackage{courier}        

\usepackage{makeidx}         
\usepackage{graphicx}        
\usepackage[bottom]{footmisc}

\usepackage{latexsym}
\usepackage{amsmath}
\usepackage{amsfonts}
\usepackage{amssymb}
\usepackage{bm}

\usepackage{url}
\usepackage{algorithm}
\usepackage{algorithmic}

\renewenvironment{proof}{\noindent{\itshape Proof.}}{\smartqed\qed}

\spdefaulttheorem{assumption}{Assumption}{\upshape \bfseries}{\itshape}
\spdefaulttheorem{algo}{Algorithm}{\upshape \bfseries}{\itshape}

\DeclareSymbolFont{bbold}{U}{bbold}{m}{n}
\DeclareSymbolFontAlphabet{\mathbbold}{bbold}

\begin{document}

\title*{ON QUASI-ENERGY-SPECTRA, PAIR CORRELATIONS OF SEQUENCES AND ADDITIVE COMBINATORICS}
\titlerunning{Quasi-enery-spectra, pair correlations, and additive combinatorics}

\author{Ida Aichinger \and Christoph Aistleitner \and Gerhard Larcher}

\institute{
 Ida Aichinger
 \at CERN, European Organization for Nuclear Research, 385 Route de Meyrin, CH-1217 Meyrin, Switzerland\\
 \email{ida.aichinger@cern.ch}
 \and
 Christoph Aistleitner
 \at TU Graz, Institute for Analysis and Number Theory, Steyrergasse 30, 8010 Graz, Austria \\
 \email{aistleitner@math.tugraz.at}
\and
Gerhard Larcher
 \at Johannes Kepler University Linz, Institute for Financial Mathematics and Applied Number Theory, Altenbergerstra{\ss}e 69, 4040 Linz, Austria\\
 \email{gerhard.larcher@jku.at}
}

\maketitle

\index{Aichinger, Ida}
\index{Aistleitner, Christoph}
\index{Larcher, Gerhard}

\abstract{The investigation of the pair correlation statistics of sequences was initially motivated by questions concerning quasi-energy-spectra of quantum systems. However, the subject has been developed far beyond its roots in mathematical physics, and many challenging number-theoretic questions on the distribution of the pair correlations of certain sequences are still open. We give a short introduction into the subject, recall some known results and open problems, and in particular explain the recently established connection between the distribution of pair correlations of sequences on the torus and certain concepts from additive combinatorics. Furthermore, we slightly improve a result recently given by Jean Bourgain in~\cite{ALL+B}.
}

\section{INTRODUCTION}\label{sec:1}

Some of Ian Sloan's first published papers dealt with topics from mathematical physics, in particular with theoretical nuclear physics. Later he moved his area of research to applied mathematics and numerical analysis, and  in particular Ian's ground-breaking work on complexity theory, numerical integration and mathematical simulation is well-known and highly respected among the scientific community of mathematicians. The techniques developed and analyzed by Ian in these fields are often based on the use of deterministic point sets and sequences with certain ``nice'' distribution properties, a method which is nowadays widely known under the name of \emph{quasi-Monte Carlo method} (QMC). In the present paper we will combine these two topics, mathematical physics and the distribution of point sets.\\

Ian's first research paper appeared 1964 in the Proceedings of the Royal Society (London), entitled ``The method of polarized orbitals for the elastic scattering of slow electrons by ionized helium and atomic hydrogen'' \cite{Sloan}. In the same journal, but thirteen years later, Berry and Tabor published a groundbreaking paper on ``Level clustering in the regular spectrum'' \cite{BT}. This paper deals with the investigation of conservative quantum systems that are chaotic in the classical limit. More precisely, the paper deals with statistical properties of the energy spectra of these quantum systems, and Berry and Tabor conjectured that for the distribution function of the spacings between neighboring levels of a generic integrable quantum system the exponential Poisson law holds. That means, roughly speaking, the following.\\

Let $H$ be the Hamiltonian of a quantum system and let $\lambda_{1} \leq \lambda_{2} \leq \ldots$ be its discrete energy spectrum. We call the numbers $\lambda_{i}$ the \emph{levels} of this energy spectrum. If it is assumed that 
$$
\# \left\{i:~ \lambda_{i} \leq x \right\} \sim c x^{\gamma}
$$
for $x \rightarrow \infty$ and some constants $c > 0, ~\gamma \geq 1$, then consider $X_{i} := c \lambda_{i}^{\gamma}$. The Berry--Tabor conjecture now states that if the Hamiltonian is classically integrable and ``generic'', then the $X_{i}$ have the same local statistical properties as independent random variables coming from a Poisson process. Here the word ``generic'' is a bit vague; it essentially means that one excludes the known obvious (and less obvious) counterexamples to the conjecture. For more material on energy spectra of quantum systems and the following two concrete examples see the original paper of Berry and Tabor \cite{BT} as well as \cite{Bleher,Casati,Marklof} and Chapter 2 in \cite{Bogo}. For a survey on the Berry--Tabor conjecture see \cite{mark2}.\\

Two basic examples of quantum systems are the two-dimensional ``harmonic oscillator'' with Hamiltonian 
$$
H = p^{2}_{x} + p_{y}^{2} + w^{2} \left(x^{2}+y^{2}\right)
$$
and the ``boxed oscillator''. This is a particle constrained by a box in $x$-direction and by a harmonic potential in $y$-direction; the Hamiltonian in this case is given by 
$$
H = -p^{2}_{x} - p^{2}_{y} + w^{2} y^{2}.
$$
The investigation of the distribution of the energy levels in these two examples leads to the investigation of the pair correlation statistics of certain sequences $\left(\theta_{n}\right)_{n \geq 1}$ in the unit interval. More specifically, one is led to study the pair correlations of the sequence $\left(\left\{n \alpha\right\}\right)_{n \geq  1}$ in the case of the $2$-dimensional harmonic oscillator, and the pair correlations of the sequence $\left(\left\{n^{2} \alpha\right\}\right)_{n \geq 1}$ in the case of the boxed oscillator; here, and in the sequel, we write $\{ \cdot \}$ for the fractional part function. In particular, for these sequences one is led to study the quantity $R_2$, which is introduced below.\\

Let $(\theta_n)_{n \geq 1}$ be a sequence of real numbers in $[0,1]$, and let $\| \cdot \|$ denote the distance to the nearest integer. For every interval $[-s,s]$ we set
$$
R_2 \big( [-s,s],(\theta_n)_{n \geq 1}, N \big) = \frac{1}{N} \# \left\{ 1 \leq j \neq k \leq N: \left\|\theta_j - \theta_k \right\| \leq \frac{s}{N} \right\}.
$$
The subscript ``2'' of ``$R_2$'' refers to the fact that these are the \emph{pair} correlations, that is, the correlations of order 2 -- in contrast to triple correlations or correlations of even higher order. Note that the average spacing between two consecutive elements of $\{\theta_1, \dots, \theta_N\}$ (understood as a point set on the torus) is $1/N$, and thus for an ``evenly distributed'' point set one would expect to find roughly $2s$ other points within distance $[-s/N,s/N]$ around a given point $\theta_j$, causing $R_2([-s,s],(\theta_n)_{n \geq 1}, N)$ to be approximately $2s$ for such a point set (after summing over all elements of the point set and then normalizing with division by $N$). Actually, for a sequence of independent, $[0,1]$-uniformly distributed random variables $\theta_1, \theta_2, \dots$ one can easily show that for every $s \geq 0$ we have
$$
R_2 ( [-s,s],(\theta_n)_{n \geq 1}, N) \to 2s,
$$
almost surely. If this asymptotic relation holds for the distribution of pair correlations of a certain sequence we say that the distribution of the pair correlations is asymptotically \emph{Poissonian}. Informally speaking, a sequence whose distribution of the pair correlations is asymptotically Poissonian may be seen as a sequence showing ``random'' behavior, and the investigation of the asymptotic distribution of the pair correlations of a deterministic sequence may be seen as studying the pseudo-randomness properties of this sequence.\\

The systematic investigation of the asymptotic distribution of the pair correlation of sequences on the torus (motivated by the applications in quantum physics) was started by Rudnick and Sarnak in \cite{R-S} for the case of sequences of the form $\left(\left\{n^{d} \alpha\right\}\right)_{n \geq 1}$ for integers $d \geq 1$. In the case $d=1$, the distribution of the pair correlations is \emph{not} asymptotically Poissonian (independent of the value of $\alpha$); this was remarked for example in \cite{R-S} with a hint to the well-known \emph{Three Distance Theorem}, which goes back to {\'S}wierczkowski and S\'os~\cite{sos}. For $d \geq 2$ the distribution of the pair correlations is asymptotically Poissonian for almost all $\alpha$, which has been proved by Rudnick and Sarnak~\cite{R-S}. The case $d=2$ (which corresponds to the energy levels of the \emph{boxed oscillator}) has received particular attention; see for example~\cite{hp,mse,rsz,td}. A generalization from $(\{n^d \alpha\})_{n \geq 1}$ 
to the case of $(\{a(n) \alpha\})_{n \geq 1}$ with $a(x) \in \mathbb{Z}[x]$ is obtained in~\cite{bzp}; again the pair correlations are asymptotically Poissonian for almost all $\alpha$, provided that the degree of $a(x)$ is at least 2. Another case which has been intensively investigated is that of  $(\left\{a(n) \alpha \right\})_{n \geq 1}$ for $(a(n))_{n \geq 1}$ being a \emph{lacunary} sequence; see for example~\cite{bpt,clz,RZ}.\\

In \cite{ALL+B} a general result was proved which includes earlier results (polynomial sequences, lacunary sequences, and sequences satisfying certain Diophantine conditions) and gives a unifying explanation. This result links the distribution of the pair correlations of the sequence $\left(\left\{a(n) \alpha\right\}\right)_{n \geq 1}$ to the additive energy of the truncations of the integer sequence $\left(a(n)\right)_{n \geq 1}$, a well-known concept from additive combinatorics which has been intensively studied. Recall that the additive energy $E(A)$ of a set of real numbers $A$ is defined as
\begin{equation}
E(A) := \sum_{a+b=c+d} 1,
\end{equation}
where the sum is extended over all quadruples (a, b, c, d) $\in A^{4}$. Trivially one has the estimate $\left|A\right|^{2} \leq E(A) \leq \left|A\right|^{3}$, assuming that the elements of $A$ are distinct. The additive energy of sequences has been extensively studied in the combinatorics literature. We refer the reader to \cite{Tao} for a discussion of its properties and applications. To simplify notations, in the sequel whenever a sequence $A:= \left(a(n)\right)_{n \geq 1}$ is fixed we will abbreviate $R_{2} \left(s, \alpha, N \right)$ for $R_{2} \left(\left[-s, s\right], \left( \{a (n) \alpha\}\right)_{n \geq 1}, N\right)$. Furthermore we will let $A_{N}$ denote the first $N$ elements of $A$. The result states that if the truncations $A_{N}$ of an integer sequence $A$ satisfy $E\left(A_{N}\right) \ll N^{3-\varepsilon}$ for some $\varepsilon > 0$, then $\left( \{a (n) \alpha \} \right)_{n \geq 1}$ has (asymptotically) Poissonian pair correlations for almost all $\alpha$. More precisely, the following theorem is true.
\begin{theorem} \label{th1}
Let $(a(n))_{n \geq 1}$ be a sequence of distinct integers, and suppose that there exists a fixed constant $\varepsilon > 0$ such that
\begin{equation}
E \left(A_{N}\right) \ll N^{3- \varepsilon}\quad\text{as}~N \rightarrow \infty.
\end{equation}
Then for almost all $\alpha$ one has
\begin{equation}
R_{2} \left(s, \alpha, N \right) \rightarrow 2s \quad \text{as}~N \rightarrow \infty
\end{equation}
for all $s \geq 0$.
\end{theorem}

Note that the condition of Theorem \ref{th1} is close to optimality, since by the trivial upper bound we always have $E \left(A_{N}\right) \leq N^{3}$; thus an arbitrarily small power savings over the trivial upper bound assures the ``quasi-random'' behavior of the pair correlations of $(\{a(n) \alpha\})_{n \geq 1}$. On the other hand, in \cite{ALL+B} Bourgain showed the following negative result.

\begin{theorem} \label{thm:B}
If $E\left(A_{N}\right)= \Omega \left(N^{3}\right)$, then there exists a subset of $[0,1]$ of positive measure such that for every $\alpha$ from this set the pair correlations of the sequence $\left(\{ a(n) \alpha\}\right)_{n \geq 1}$ are \emph{not} asymptotically Poissonian.
\end{theorem}

We conjecture that actually even the following much stronger assertion is true.

\begin{conjecture}
\textit{If $E\left(A_{N}\right) = \Omega \left(N^{3}\right)$ there is \textbf{no} $\alpha$ for which the pair correlations of the sequence $\left(\left\{a(n) \alpha\right\}\right)_{n \geq 1}$ are Poissonian.}
\end{conjecture}

In this paper we will prove a first partial result which should support this conjecture. However, before stating and discussing our result (which will be done in Section \ref{sec2} below), we want to continue our general discussion of pair correlation problems. As one can see from the previous paragraphs, the metric theory of pair correlation problems on the torus is relatively well-understood. In contrast, there are only very few corresponding results which hold for a \emph{specific} value of $\alpha$. The most interesting case is that of the sequence $(\{n^2 \alpha\})_{n \geq 1}$, where it is assumed that there is a close relation between Diophantine properties of $\alpha$ and the pair correlations distribution. For example, it is conjectured that for a quadratic irrational $\alpha$ this sequence has a pair correlations distribution which is asymptotically Poissonian; however, a proof of this conjecture seems to be far out of reach. A first step towards a proof of the conjecture was made by Heath-Brown \cite{hp}, whose method requires bounds on the number of solutions of certain quadratic congruences; this topic was taken up by Shparlinski \cite{sp1,sp2}, who obtained some improvements, but new ideas seem to be necessary for further steps toward a solution of the conjecture.\\

It should also be noted that the investigation of pair correlation distributions is not restricted to sequences of the torus. For example, consider a positive definite quadratic form $P(x,y) = \alpha x^2 + \beta xy + \gamma y^2$, and its values at the integers $(x,y) = (m,n) \in \mathbb{Z}^2$. These values form a discrete subset of $\mathbb{R}$, and one can study the pair correlations of those numbers contained in a finite window $[0,N]$. See for example \cite{sarnak,van}. Another famous occurrence of the pair correlation statistics of an unbounded sequence in $\mathbb{R}$ is in Montgomery's pair correlation conjecture for the normalized spacings between the imaginary parts of zeros of the Riemann zeta function. The statement of the full conjecture is a bit too long to be reproduced here; we just want to mention that it predicts a distribution of the pair correlations which is very different from ``simple'' random behavior. For more details see Montgomery's paper \cite{mont}. There is a famous story related to Montgomery's conjecture; he met the mathematical physicist Freeman Dyson at tea time at Princeton, where Freeman Dyson identified Montgomery's conjectured distribution as the typical distribution of the spacings between normalized eigenvalues of large random Hermitian matrices -- an observation which has led to the famous (conjectural) connection between the theory of the Riemann zeta function and random matrix theory. The whole story and more details can be found in \cite{bourg}.

\section{New results} \label{sec2}

In the sequel we give a first partial result towards a solution of the conjecture made above. Before stating the result we introduce some notations, and explain the background from additive combinatorics. For $v \in \mathbb{Z}$ let $A_{N} (v)$ denote the cardinality of the set
$$
\left\{\left(x,y\right) \in \left\{1, \ldots, N\right\}^{2}, x \neq y : a(x) - a(y) = v \right\}.
$$
Then
\begin{equation} \label{equ_a}
E \left(A_{N}\right) = \Omega \left(N^{3}\right)
\end{equation}
is equivalent to
\begin{equation} \label{equ_b}
\sum_{v \in \mathbb{Z}} A^{2}_{N} (v) = \Omega \left(N^{3}\right),
\end{equation}
which implies that there is a $\kappa > 0$ and positive integers $N_{1} < N_{2} < N_{3} < \ldots$ such that
\begin{equation} \label{equ_c2}
\sum_{v \in \mathbb{Z}} A^{2}_{N_{i}} (v) \geq \kappa N^{3}_{i}, \qquad i = 1,2,\dots.~\\
\end{equation}

It will turn out that sequences $(a(n))_{n \geq 1}$ satisfying \eqref{equ_a} have a strong linear substructure. From \eqref{equ_c2} we can deduce by the Balog--Szemeredi--Gowers-Theorem (see \cite{Ba+Sz} and \cite{Gow1}) that there exist constants $c, C > 0$ depending only on $\kappa$ such that for all $i=1,2,3,\ldots$ there is a subset $A_{0}^{(i)} \subset \left(a\left(n\right)\right)_{1 \leq n \leq N_{i}}$ such that 
$$
\left|A_{0}^{(i)}\right| \geq c N_{i} \qquad \text{and} \qquad \left|A_{0}^{(i)} + A_{0}^{(i)}\right| \leq C  \left|A_{0}^{(i)}\right| \leq C N_{i}.
$$
The converse is also true: If for all $i$ for a set $A_{0}^{(i)}$ with $A_{0}^{(i)} \subset \left(a\left(n\right)\right)_{1 \leq n \leq N_{i}}$ with $\left| A_{0}^{(i)} \right| \geq c N_{i}$ we have $\left|A_{0}^{(i)} + A_{0}^{(i)}\right| \leq C \left|A_{0}^{(i)}\right|$, then 
$$
\sum_{v \in \mathbb{Z}} A^{2}_{N_{i}} (v) \geq \frac{1}{C} \left|A_{0}^{(i)}\right|^{3} \geq \frac{c^3}{C} N_{i}^{3}
$$
and consequently $\sum_{v \in \mathbb{Z}} A^{2}_{N} (v)= \Omega \left(N^{3}\right)$ (this an elementary fact, see for example Lemma~1~(iii) in \cite{vflev}.)\\

Consider now a subset $A_{0}^{(i)}$ of $\left(a (n)\right)_{1 \leq n \leq N_{i}}$ with 
$$
\left|A_{0}^{(i)}\right| \geq c N_{i} \qquad \text{and} \qquad \left|A_{0}^{(i)} + A_{0}^{(i)} \right| \leq C \left|A_{0}^{(i)}\right|.
$$
By the theorem of Freiman (see \cite{Frei}) there exist constants $d$ and $K$ depending only on $c$ and $C$, i.e. depending only on $\kappa$ in our setting, such that there exists a \emph{$d$-dimensional arithmetic progression} $P_{i}$ of size at most $K N_{i}$ such that $A_{0}^{(i)} \subset P_{i}$. This means that $P_i$ is a set of the form 
\begin{equation}\label{equ_c}
P_{i} := \left\{ \left. b_{i} + \sum^{d}_{j=1} r_{j} k_{j}^{(i)} \right|  0 \leq r_{j} < s_{j}^{(i)} \right\},  
\end{equation}
with $b_{i}, k_{1}^{(i)}, \ldots, k^{(i)}_{d}, s_{1}^{(i)}, \ldots, s_{d}^{(i)} \in \mathbb{Z}$ and such that $s_{1}^{(i)} s_{2}^{(i)} \ldots s_{d}^{(i)} \leq K N_{i}$.\\

In the other direction again it is easy to see that for any set $A_{0}^{(i)}$ of the form \eqref{equ_c} we have
$$
\left|A_{0}^{(i)} + A_{0}^{(i)}\right| \leq 2^{d} K N_{i}.
$$

Based on these observations we make the following definition:

\begin{definition} \label{def_a}
Let $\left(a\left(n\right)\right)_{n \geq 1}$ be a strictly increasing sequence of positive integers. We call this sequence {\em quasi-arithmetic of degree} $\mathbf{d}$, where $d$ is a positive integer, if there exist constants $C,K > 0$ and a strictly increasing sequence $\left(N_{i}\right)_{i \geq 1}$ of positive integers such that for all $i \geq 1$ there is a subset $A^{(i)} \subset \left(a\left(n\right)\right)_{1 \leq n \leq N_{i}}$ with $\left|A^{(i)}\right| \geq C N_{i}$ such that $A^{(i)}$ is contained in a $d$-dimensional arithmetic progression $P^{(i)}$ of size at most $K N_{i}$.~\\
\end{definition}

The considerations above show that a sequence $\left(a\left(x\right)\right)_{x \geq1}$ is quasi-arithmetic of some degree $d$ if and only if it satisfies \eqref{equ_b}.\\

So our conjecture is equivalent to
\begin{conjecture}
\textit{If $\left(a(n)\right)_{n \geq 1}$ is a quasi-arithmetic sequence of integers then there is \textbf{no} $\alpha$ such that the pair correlations of $\left(\left\{a(n) \alpha\right\}\right)_{x \geq 1}$ are asymptotically Poissonian.}
\end{conjecture}
In the remaining part of this paper we will prove a theorem which slightly improvements the Theorem~\ref{thm:B} of Bourgain for the subclass of sequences $\left(a(n)\right)_{n \geq 1}$ which are quasi-arithmetic of degree 1.
\begin{theorem} \label{thm:C}
If the sequence of integers $\left(a(n)\right)_{n \geq 1}$ is quasi-arithmetic of degree 1, then the set of $\alpha$'s for which the distribution of the pair correlations of $\left(\left\{a(n) \alpha\right\}\right)_{n \geq 1}$ is not asymptotically Poissonian has full measure.
\end{theorem}
\textbf{Remark.}
\textit{The class of quasi-arithmetic sequences $\left(a\left(n\right)\right)_{n\geq 1}$ of degree 1 contains all strictly increasing sequences with positive upper density, i.e.
$$
\underset{N \rightarrow \infty}{\lim \sup}~\frac{1}{N} \underset{m \in \left\{a \left(n\right) \left| \right. n \geq 1 \right\}}{\sum^{N}_{n=1}} 1 > 0.
$$
In particular this class contains all strictly increasing sequences which are bounded above by a linear function.}\\

We will first state two auxiliary results in Section~\ref{sec:2}, and then give the proof of Theorem~\ref{thm:C} in Section~\ref{sec:3}.

\section{Auxiliary results} \label{sec:2}

\begin{lemma} \label{lem_a}
Let $\left(\lambda_{n}\right)_{n \geq 1}$ be a strictly increasing sequence of positive integers. Let $\mu_{n}$ be the number of fractions of the form $j \lambda_{n}^{-1} \left(0 < j< \lambda_{n}\right)$ which are not of the form $k \lambda_{q}^{-1}$ with some $q < n$ and $ k < \lambda_q$. Furthermore, let $\left(\psi_{n}\right)_{n \geq 1}$ be a non-increasing sequence of positive reals such that $\sum^{\infty}_{n=1} \psi_{n} = \infty$ and with the following property~(*):\\

There exists a sequence $\left(\tau_{n}\right)_{n \geq 1}$ of positive reals tending monotonically to zero, but so slowly that $\sum^{\infty}_{n=1} \psi_{n} \tau_{n}$ still diverges, and such that there exist a constant $c > 0$ and infinitely many positive integers $N$ with 
$$
\sum^{N}_{n=1} \mu_{n} \lambda_{n}^{-1} \psi_{n} \tau_{n} > c \sum^{N}_{n=1} \psi_{n} \tau_{n}.
$$
Then -- if (*) holds -- for almost all $\theta \in \mathbb{R}$ there exist infinitely many positive integers $n$, and integers $m$, such that
$$
0 \leq \lambda_{n} \theta - m < \psi_{n}.
$$
\end{lemma}

\begin{proof}
This lemma is essentially the divergence part of Theorem IV in \cite{Cassels}. It is shown there that the assertion of our Lemma~\ref{lem_a} is true under the slightly stronger condition that $\left(\psi_{n}\right)_{n \geq 1}$ -- as in our Lemma -- is a non-increasing sequence of positive reals with $\sum^{\infty}_{n=1} \psi_{n} = \infty$, and that $\left(\lambda_{n}\right)_{n \geq 1}$ satisfies 
$$
\underset{N \rightarrow \infty}{\lim \inf}~\frac{1}{N} \sum^{N}_{n=1} \mu_{n} \lambda_{n}^{-1} > 0.
$$
If we follow the proof of Theorem IV in \cite{Cassels} line by line we see that our slightly weaker condition (*) also is sufficient to obtain the desired result. In fact replacing Cassel's condition by our condition (*) is relevant only in the proof of Lemma~3 in \cite{Cassels}, which is an auxiliary result for the proof of Theorem IV in \cite{Cassels}.
\end{proof}

\begin{lemma} \label{lem_b}
For all $\delta>0$ there is a positive constant $c{(\delta)} > 0$, such that for every infinite subset $A$ of $\mathbb{N}$ with 
$$
\underline{d}(A) := \underset{N \rightarrow \infty}{\lim \inf} ~\frac{1}{N} \# \left\{n \leq N \left| \right. n \in A \right\} > \delta
$$
we have
$$
\underset{N \rightarrow \infty}{\lim \inf}~\frac{1}{N} \underset{n \in A}{\sum_{n \leq N}} \frac{\varphi (n)}{n} \geq c{\left(\delta\right)}.
$$
Here $\varphi$ denotes the Euler totient function.
\end{lemma}

\begin{proof}
Let 
$$
B (t) := \underset{N \rightarrow \infty}{\lim}~\frac{1}{N} \left|\left\{n \leq N \left| \frac{n}{\varphi (n)} \geq t\right\} \right.\right|.
$$
Then by the main theorem in \cite{Wein} the limit $B(t)$ exists and satisfies
$$
B(t) = \exp \left(-e^{t e^{-\gamma}} \left(1+ O \left(t^{-2}\right)\right) \right)
$$
for $t$ to infinity and with $\gamma$ denoting Euler's constant. Here, and in the sequel, we write $\exp(x)$ for $e^x$. (\emph{Note}: An earlier version of the paper contained an error in the formula for the function $B(t)$ above. The same error also appears in the printed version of the paper. This error also requires corrections in the remaining part of the proof of the lemma, which appears in corrected form below.)\\

So there is a constant $L > 0$ such that
$$
B(t) \leq \exp \left(-e^{t e^{-\gamma}} \left(1 - \frac{L}{t^{2}}\right) \right)
$$
for all $t \geq 1$. Hence
$$
B(t) \leq \exp \left(-\frac{e^{t e^{- \gamma}}}{2} \right)
$$
for all $t \geq \max \left(1, \sqrt{2 L}\right)$. Now assume that $\delta > 0$ is so small that 
$$t_{0} :=  e^{\gamma}  \log \left(-2\log \frac{\delta}{4} \right) > \max \left(1, \sqrt{2 L}\right).
$$
Note that it suffices to prove the lemma for such $\delta$ . We have
$$
B \left(t_{0}\right) = \underset{N \rightarrow \infty}{\lim}~\frac{1}{N} \left|\left\{n \leq N \left| \frac{n}{\varphi (n)} \geq t_0 \right\} \right. \right|
$$
and
$$
B\left(t_{0}\right) \le \exp \left(-\frac{e^{t_0 e^{- \gamma}}}{2} \right) = \frac{\delta}{4}.
$$
Hence there exists $N_{0}$ such that for all $N \geq N_{0}$
$$
\frac{1}{N} \left|\left\{n \leq N \left| \frac{n}{\varphi (n)} \geq t_{0} \right\} \right. \right| \leq \frac{\delta}{3}.
$$
Therefore, since $\underline{d}(A) > \delta$, for all sufficiently large $N$ we have
$$
\frac{1}{N} \left|\left\{n \leq N, n \in A \left| \frac{n}{\varphi (n)} \leq t_{0}\right\} \right. \right| \geq \frac{\delta}{3}
$$
and consequently also
$$
\frac{1}{N} \underset{n \in A}{\sum_{n \leq N}} \frac{\varphi (n)}{n} \geq \frac{\delta}{3} \frac{1}{t_{0}} =: c (\delta) >0.
$$
\end{proof}

\section{Proof of Theorem~\ref{thm:C}} \label{sec:3}

Let $\left(a(n)\right)_{n \geq 1}$ be quasi-arithmetic of degree one and let $C, K > 0, \left(N_{i}\right)_{i \geq 1}, \left(A^{(i)}\right)_{i \geq 1}$ and $\left(P^{(i)}\right)_{i \geq 1}$ be as described in Definition~\ref{def_a}. In the sequel we will define inductively a certain strictly increasing subsequence $\left(M_{l}\right)_{l \geq 1}$ of $\left(N_{i}\right)_{i\geq1}$.\\

Set $M_{1} := N_{1}$ and assume that $M_{1}, M_{2}, \ldots, M_{l-1}$ already are defined. If $M_{l} = N_{i_{l}}$ (where $i_{l}$ still has to be defined) to simplify notations we write $A_{l} := A^{\left(i_{l}\right)},~ P_{l} := P^{\left(i_{l}\right)}.$\\

We set
$$
P_{l} := \left\{a_{l} + r \kappa_{l} \left| \right. 0 \leq r < K M_{l}\right\}
$$
and
$$
A_{l} := \left\{ \left. a_{l} + r_{j}^{(l)} \kappa_{l} \right|  j=1,2,\ldots, s_{l}  \right\}
$$
with certain fixed $r_{j}^{(l)}$ with $1 \leq r_{1}^{(l)} < r_{2}^{(l)} < \ldots < r_{s_{l}}^{(l)} < K M_{l}$
and $s_{l} \geq C M_{l}$. Of course we have $s_{l} < K M_{l}$.\\

We consider
$$
V_{l} := \left\{ \left. \left(r_{i}^{(l)} - r_{j}^{(l)} \right) \kappa_{l} \right| 1 \leq j < i \leq s_{l} \right\},
$$
the set of positive differences of $A_{l}$. Here $V_{l}$ is the set itself, whereas by $\widetilde{V}_{l}$ we will denote the same set of positive differences but counted with multiplicity (so strictly speaking $\widetilde{V}_{l}$ is a multi-set rather than a set). Hence $\left|V_{l}\right|< K M_{l}$, whereas 
$$
\left|\widetilde{V}_{l}\right|= \frac{s_{l} \left(s_{l}-1\right)}{2} \geq c_{1} M_{l}^{2}.
$$
Here and in the sequel we write $c_{i}$ for positive constants depending only on $C$ and $K$. We note that a value $u \in V_{l}$ has multiplicity at most $s_{l}$.\\

Let $x$ be the number of elements in $V_{l}$ with multiplicity at least $c_{2} M_{l}$ where $c_{2} :=~\min \left(K, \frac{c_{1}}{2 K}\right).$ Assume that $x < c_{2} M_{l}$. Then

\begin{eqnarray*}
c_{1}  M_{l}^{2} & \leq & \left|\widetilde{V}_{l}\right| \leq x  s_{l} + \left(\left|V_{l}\right| -x \right)  c_{2}  M_{l}\\
& \leq & x  K  M_{l} + \left(K  M_{l} -x\right)  c_{2} M_{l}\\
& = & M_{l} \left(x \left(K - c_{2}\right) + K  c_{2} M_{l}\right)\\
& < & M_{l}^{2} \left(c_{2}  \left(K - c_{2}\right) + K c_{2}\right)\\
& < & M_{l}^{2} c_{1},
\end{eqnarray*}
a contradiction.\\

So there are at least $c_{2}  M_{l}$ values $u \in V_{l}$ with multiplicity at least $c_{2}  M_{l}$. We take the $\frac{c_{2}}{2} M_{l}$ largest of these values and denote them by $T_{1}^{(l)} < T_{2}^{(l)} < \ldots < T_{w_{l}}^{(l)}$ with $w_{l} \geq \frac{c_{2}}{2} M_{l}$ and $T_{j}^{(l)} := R_{j}^{(l)}  \kappa_{l}$. Note that
\begin{equation} \label{equ_e}
\frac{c_{2}}{2} M_{l} \leq R_{1}^{(l)} < \ldots < R_{w_{l}}^{(l)} < K  M_{l}.
\end{equation}

Remember that we still have to choose $i_{l} > i_{l-1}$ and to define $M_{l}$ as $N_{i_{l}}$. We choose now $i_{l}$ so large that 

\begin{equation} \label{equ_f}
M_{l} > \left(\sum^{l-1}_{p=1} \sum^{w_{p}}_{q=1} T_{q}^{(p)}\right)^{2}.
\end{equation}

So altogether we have constructed a strictly increasing sequence $\lambda_{1} < \lambda_{2} < \lambda_{3} < \ldots$ of integers given by $T_{1}^{(1)} < \ldots < T^{(1)}_{w_{1}} < T_{1}^{(2)} < \ldots < T_{w_{2}}^{(2)} < T_{1}^{(3)} < \ldots$.\\

Furthermore we define a decreasing sequence $\left(\psi_{n}\right)_{n \geq 1}$ of positive reals in the following way. If $\lambda_{n}$ is such that $T_{1}^{(l)} \leq \lambda_{n} \leq T_{w_{l}}^{(l)}$, then $\psi_{n} := \frac{1}{M_{l}}$.\\

Obviously we have
$$
\underset{n \rightarrow \infty}{\lim} \psi_{n} = 0
$$ 
and
$$
\sum^{\infty}_{n=1} \psi_{n} \geq \sum^{\infty}_{l=1} w_{l}  \frac{1}{M_{l}} \geq \sum^{\infty}_{l=1} \frac{c_{2}}{2} M_{l}  \frac{1}{M_{l}} = \infty.
$$
We will show below that $\left(\lambda_{n}\right)$ and $\left(\psi_{n}\right)$ satisfy the condition (*) of Lemma~\ref{lem_a}.\\

We choose $N:= w_{1} + \ldots + w_{l}$ and first estimate $\sum_{n \leq N} \mu_{n} \lambda_{n}^{-1} \psi_{n}$ from below (for the definition of $\mu_{n}$ see Lemma~\ref{lem_a}). We have
$$
\sum_{n \leq N} \mu_{n} \lambda_{n}^{-1} \psi_{n} \geq \sum^{N}_{n=N-w_{l}+1} \mu_{n} \lambda_{n}^{-1} \psi_{n}.
$$
In the following we estimate $\mu_{n}$ from below for $n$ with $N-w_{l} +1 \leq n \leq N,$ i.e., $\lambda_{n} = T_{i}^{(l)} = R_{i}^{(l)}  \kappa_{l}$ for some $i$ with $1 \leq i \leq w_{l}$.\\

Consider first $\lambda_{q}$ with $q \leq w_{1} + \ldots + w_{l-1}$. Then the number of $j$ with $0 \leq j < \lambda_{n}$ such that $j  \lambda_{n}^{-1}$ is of the form $k  \lambda_{q}^{-1}$ with $0 \leq k < \lambda_{q}$ trivially is at most $\lambda_{q}$.\\

Now consider $\lambda_{q}$ with $q > w_{1} + \ldots + w_{l-1}$ and $\lambda_{q} < \lambda_{n}$, i.e.,
$$
\lambda_{q} = T_{h}^{(l)} = R_{h}^{(l)}  \kappa_{l}
$$
for some $h$ with $1 \leq h < i$. Then the number of $j$ with $0 \leq j < \lambda_{n}$ such that $j  \lambda_{n}^{-1}$ is \textbf{not} of the form $k  \lambda_{q}^{-1}$ with $0 \leq k < \lambda_{q}$, i.e., such that
\begin{eqnarray*}
\frac{j}{\lambda_{n}} = \frac{k}{\lambda_{q}} & \Leftrightarrow & \frac{j}{R_{i}^{(l)}  \kappa_{l}} = \frac{k}{R_{h}^{(l)}  \kappa_{l}}\\
& \Leftrightarrow & \frac{j}{R_{i}^{(l)}} = \frac{k}{R_{h}^{(l)}}
\end{eqnarray*}
does \textbf{not} hold, is at least $ \varphi \left(R_{i}^{(l)}\right)  \kappa_{l}$. Hence by \eqref{equ_e} and by \eqref{equ_f}
\begin{eqnarray*}
\mu_{n} & \geq & \varphi \left(R_{i}^{(l)}\right)  \kappa_{l} - \sum_{q=1}^{w_{1} + \ldots + w_{l-1}} \lambda_{q}\\
& \geq & \varphi \left(R_{i}^{(l)}\right)  \kappa_{l} - \sqrt{M_{l}} \geq \frac{1}{2} \varphi \left(R_{i}^{(l)}\right)  \kappa_{l}
\end{eqnarray*}
for all $l$ large enough, say $l \geq l_{0}$ (note that $R_i^{(l)} \ge \frac{c_2}{2} M_l$).\\

Therefore for $ l \ge l_0$
\begin{eqnarray} \label{equ_g}
\sum_{n \leq N} \mu_{n} \lambda_{n}^{-1} \psi_{n} & \geq & \sum^{N}_{n=N-w_{l}+1} \mu_{n} \lambda_{n}^{-1} \psi_{n}\\ \nonumber
& \geq & \frac{1}{M_{l}} \sum_{i=1}^{w_{l}} \frac{1}{2} \varphi \left(R_{i}^{(l)}\right)  \kappa_{l}  \frac{1}{R_{i}^{(l)}  \kappa_{l}}\\ \nonumber
& = & \frac{1}{2 M_{l}} \sum_{i=1}^{w_{l}} \frac{\varphi \left(R_{i}^{(l)}\right)}{R_{i}^{(l)}}.
\end{eqnarray}
Later on we will use the same chain of inequalities starting from the second expression in \eqref{equ_g}.\\

We recall that $w_{l} \geq \frac{c_{2}}{2} M_{l}$, and $R_{i}^{(l)} \leq K  M_{l}$ for all $i=1,\ldots, w_{l}$. Hence $R_{1}^{(l)}, \ldots, R_{w_{l}}^{(l)}$ form a subset of $\left\{1,2,\ldots, K  M_{l}\right\}$ of density at least $c_{3} := \frac{c_{2}}{2 K}$. Hence by Lemma~\ref{lem_b} we have for $l$ large enough and with $c$ from Lemma 2 that
\begin{equation} \label{equ_h}
\sum_{n \leq N} \mu_{n} \lambda_{n}^{-1} \psi_{n} \geq \frac{K}{2}  c \left(\frac{c_2}{2 K}\right) =: c_{4} > 0.
\end{equation}
This holds for all $N = w_{1} + \ldots + w_{l}$ and all $l \geq l_{0}$.\\

Finally we have to choose the function $\left(\tau_{n}\right)_{n \geq 1}$ from condition (*) in Lemma \ref{lem_a} in a suitable way. If $\lambda_{n}$ is such that $T_{1}^{(l)} \leq \lambda_{n} \leq T_{w_{l}}^{(l)}$, i.e., if $\psi_{n} = \frac{1}{M_{l}}$, then we set $\tau_{n} := \frac{1}{l}$. Then 
$$
\sum^{\infty}_{n=1} \psi_{n} \tau_{n} \geq \sum^{\infty}_{l=1} w_{l}  \frac{1}{M_{l}}  \frac{1}{l} \geq \sum^{\infty}_{l=1} \frac{c_{2}}{2} M_{l}  \frac{1}{M_{l}}  \frac{1}{l} = \infty.
$$
Finally, on the one hand for all $N=w_{1} + \ldots + w_{l}$ we have by \eqref{equ_g} and \eqref{equ_h} that
$$
\sum_{n \leq N} \mu_{n} \lambda_{n}^{-1} \psi_{n} \tau_{n} \geq \sum^{l}_{l' = l_{0}} c_{4}  \frac{1}{l'} \geq c_{5}  \log l
$$
for all $l \geq l_{0}$.\\

On the other hand we have
$$
\sum_{n \leq N} \psi_{n} \tau_{n} \leq \sum^{l}_{l'=1} w_{i_{l}}  \frac{1}{M_{l}}  \frac{1}{l} = \sum^{l}_{l'=1} K  \frac{1}{l} \leq c_{6}  \log l.
$$
Consequently
$$
\sum_{n \leq N} \mu_{n} \lambda_{n}^{-1} \psi_{n} \tau_{n} \geq c_{5}  \log l \geq \frac{c_{5}}{c_{6}}  \sum_{n \leq N} \psi_{n} \tau_{n}
$$
and the conditions of Lemma~\ref{lem_a} are satisfied for $\left(\lambda_{n}\right)_{n \geq 1}$ and $\left(\psi_{n}\right)_{n \geq 1}$. We conclude from Lemma 1 that for almost all $\alpha$ there exist infinitely many $n$ such that $\left\|\lambda_{n} \alpha \right\| \leq \psi_{n}$ holds. Let such an $\alpha$ be given, and let $n_{1} < n_{2} < n_{3} < \ldots$ be such that $\left\|\lambda_{n_{i}} \alpha \right\|\leq \psi_{n_{i}}$ for all $i=1,2,3,\ldots$. For any $n_{i}$ let $l \left(n_{i}\right)$ be defined such that $w_{1} + w_{2} + \ldots + w_{l \left(n_{i}\right)-1} < n_{i} \leq w_{1} + w_{2} + \ldots + w_{l \left(n_{i}\right)}$, then $\psi_{n_{i}} = \frac{1}{M_{l \left(n_{i}\right)}}$, hence
$$
0 \leq \left\|\lambda_{n_{i}} \alpha \right\|  M_{l \left(n_{i}\right)} < 1
$$
for all $i$.\\

Let $\rho$ with $0 \leq \rho \leq 1$ be a limit point of $\left(\left\|\lambda_{n_{i}} \alpha \right\|  M_{l \left(n_{i}\right)}\right)_{i = 1,2, \ldots}.$ We distinguish now between two cases.\\

First case: $\rho = 0$.\\

Then there exists a subsequence $m_{1} < m_{2} < m_{3} < \ldots$ of $n_{1} < n_{2} < n_{3} < \ldots$ such that
$$
0 \leq \left\|\lambda_{m_{i}} \alpha \right\| < \frac{1}{M_{l \left(m_{i}\right)}}  \frac{c_{2}}{4 K^{2}}
$$
for all $i$. $\lambda_{m_{i}}$ is an element of $V_{l \left(m_{i}\right)}$ with multiplicity at least $c_{2}  M_{l \left(m_{i}\right)}$. Hence there exist at least $c_{2}  M_{l \left(m_{i}\right)}$ pairs $\left(p, q\right)$ with
$$
1 \leq p < q \leq s_{l \left(m_{i}\right)} < K  M_{l \left(m_{i}\right)}
$$
and 
$$
\left\|\left\{a \left(q\right) \alpha \right\} - \left\{a \left(p\right) \alpha \right\}\right\| < \frac{1}{M_{l \left(m_{i}\right)}}  \frac{c_{2}}{4 K^{2}}.
$$
Let now $s=\frac{c_{2}}{4 K}$ then for all $M = K  M_{l \left(m_{i}\right)}$ we have
$$
\frac{1}{M} \# \left\{1 \leq p \neq q \leq M: \left\|\left\{a \left(q\right) \alpha \right\} - \left\{a \left(p\right) \alpha \right\}\right\| \leq \frac{s}{M}\right\} \geq \frac{c_{2}}{K} = 4 s,
$$
and hence
$$
R_{2} \left(\left[-s, s\right], \alpha, M \right) \not\rightarrow 2s.
$$

Second case: $\rho > 0$.\\

Let $\varepsilon := \min \left(\frac{\rho}{2}, \frac{c_{2}}{8 K^{2}}\right) > 0$. Then there exists a subsequence $m_{1}< m_{2}< m_{3} < \ldots$ of $n_{1} < n_{2} < n_{3} < \ldots$ such that
$$
0 \leq \left|M_{l \left(m_{i}\right)}  \left\|\lambda_{m_{i}} \alpha \right\| - \rho \right| < \varepsilon
$$
for all $i$. Hence there exist at least $c_{2}  M_{l \left(m_{i}\right)}$ pairs $\left(p, q\right)$ with $1 \leq p < q \leq s_{l \left(m_{i}\right)} < K  M_{l \left(m_{i}\right)}$ and
$$
\left\|\left\{a \left(q\right) \alpha \right\} - \left\{a \left(p\right) \alpha \right\}\right\| \in \left[\frac{\rho - \varepsilon}{M_{l \left(m_{i}\right)}}, \frac{\rho + \varepsilon}{M_{l \left(m _{i}\right)}} \right].
$$
Let $s_{1} := K  \left(\rho-\varepsilon \right)$ and $s_{2} := K  \left(\rho + \varepsilon\right)$, then $s_{2} - s_{1} = 2 K \varepsilon \leq \frac{c_{2}}{4 K}$. Let for $M:= K  M_{l \left(m _{i}\right)}$ and $j=1,2$:
$$
\Lambda^{(j)} := \frac{1}{M} \# \left\{1 \leq p \neq q \leq M : \left\|\left\{a \left(q\right) \alpha \right\} - \left\{a \left(p\right) \alpha \right\}\right\| \leq \frac{s_{j}}{M}\right\}.
$$
Then $\Lambda^{(2)} - \Lambda^{(1)} \geq \frac{1}{M}  c_{2}  \frac{M}{K} = \frac{c_{2}}{K}$. Hence at least one of
\begin{eqnarray*}
& & \left|\Lambda^{(2)} - 2s_{2} \right| \geq \frac{c_{2}}{8 K} \quad \mbox{or} \\
& & \left|\Lambda^{(1)} - 2s_{1} \right| \geq \frac{c_{2}}{8 K} \quad \mbox{holds},
\end{eqnarray*}
since otherwise
\begin{eqnarray*}
\frac{c_{2}}{2 K} & \leq & \left|\Lambda^{(2)} - \Lambda^{(1)} \right| - 2 \left(s_{2}-s_{1}\right) \\
& \leq & \left|\Lambda^{(2)} - 2s_{2} - \Lambda^{(1)} + 2s_{1} \right| \leq \left|\Lambda^{(2)} -2s_{2} \right|+\left|\Lambda^{(1)} -2s_{1}\right|\\
& \leq & \frac{c_{2}}{4 K},
\end{eqnarray*}
which is a contradiction. Therefore either
$$
R_{2} \left(\left[-s_{1}, s_{1}\right], \alpha, M \right) \not\rightarrow 2s_{1} \quad \mbox{or}
$$
$$
R_{2} \left(\left[-s_{2}, s_{2}\right], \alpha, M\right) \not\rightarrow 2s_{2},
$$
which proves the theorem.

\begin{acknowledgement}
The first author is supported by CERN, European Organization for Nuclear Research, Geneva, as part of the Austrian Doctoral Student Programme.\\
The second author is supported by the Austrian Science Fund (FWF), START-project Y-901.\\
The third author is supported by the Austrian Science Fund (FWF): Project F5507-N26, which is part of the Special Research Program ``Quasi-Monte Carlo Methods: Theory and Applications''.
\end{acknowledgement}

%

\begin{thebibliography}{99.}%

\bibitem{ALL+B}
Aistleitner Ch., Larcher G., Lewko M.:
Additive energy and the Hausdorff dimension of the exceptional set in metric pair correlation problems. With an appendix by Jean Bourgain.
\newblock To appear in: Israel J. Math.

\bibitem{Ba+Sz}
Balog A., Szemer\'edi E.:
A statistical theorem of set addition.
\newblock Combinatorica, \textbf{14}, 263--268 (1994)

\bibitem{bpt}
Berkes I., Philipp W., Tichy R.:
Pair correlations and {$U$}-statistics for independent and weakly dependent random variables.
\newblock Illinois J. Math., \textbf{45}(2), 559--580 (2001)

\bibitem{BT}
Berry M. V., Tabor M.:
Level clustering in the regular spectrum.
\newblock Proc. R. Soc. (London), A 356, 375--394 (1977)

\bibitem{Bleher}
Bleher P. M.:
The energy level spacing for two harmonic oscillators with golden mean ratio of frequencies.
\newblock J. Stat. Phys. \textbf{61}, 869--876 (1990)

\bibitem{bzp}
Boca F. P., Zaharescu A.:
Pair correlation of values of rational functions (mod {$p$}).
\newblock Duke Math. J., \textbf{105}(2), 267--307 (2000)

\bibitem{Bogo}
Bogomolny E.:
Quantum and arithmetic chaos.
\newblock Proceedings of the Les Houches Winter School ``Frontiers in Number Theory, Physics and Geometry'' (2003)

\bibitem{bourg}
Bourgade P.:
Tea Time at Princeton. 
\newblock Harvard College Mathematical Review \textbf{4}, 41--53 (2012). Available at \url{http://www.cims.nyu.edu/~bourgade/papers/TeaTime.pdf}.

\bibitem{Casati}
Casati G., Guarneri I., Izrailev F. M.:
Statistical properties of the quasi-energy spectrum of a simple integrable system.
\newblock Phys. Lett. A 124, 263--266 (1987)

\bibitem{Cassels}
Cassels J. W. S.:
Some metrical theorems in Diophantine Approximation I.
\newblock Math. Proc. Cam. Phil. Soc., \textbf{46}, 209--218 (1950)


\bibitem{clz}
Chaubey S., Lanius M., Zaharescu A.:
Pair correlation of fractional parts derived from rational valued sequences.
\newblock J. Number Theory, \textbf{151}, 147--158 (2015)

\bibitem{Frei}
Freiman G. R.:
Foundations of a Structural Theory of Set Addition.
\newblock Translations of Mathematical Monographs 37, Amer. Math. Soc., Providence, USA, (1973)

\bibitem{Gow1}
Gowers W. T.:
A new proof of Szemer\'edi's theorem for arithmetic progressions of length four.
\newblock Geometric and Functional Analysis, \textbf{8}, 529--551 (1998)

\bibitem{hp}
Heath-Brown D. R.:
Pair correlation for fractional parts of {$\alpha n^2$}.
\newblock Math. Proc. Cambridge Philos. Soc., \textbf{148}(3), 385--407 (2010)

\bibitem{vflev}
Lev, V.F.:
The (Gowers--)Balog--Szemer\'edi Theorem: An Exposition. 
\newblock \url{http://people.math.gatech.edu/~ecroot/8803/baloszem.pdf}

\bibitem{Marklof}
Marklof J.:
Energy level statistics, lattice point problems and almost modular functions.
\newblock Frontiers in Number Theory, Physics and Geometry. Volume 1: On random matrices, zeta functions and dynamical systems (Les Houches, 9-21 March 2003), Springer, 163--181

\bibitem{mark2}
Marklof J.:
The Berry--Tabor conjecture. 
\newblock European Congress of Mathematics, Vol. II (Barcelona, 2000), 421--427, Progr. Math., 202, Birkh\"auser, Basel, 2001. 

\bibitem{mse}
Marklof J., Str{\"o}mbergsson A.:
Equidistribution of {K}ronecker sequences along closed horocycles.
\newblock Geom. Funct. Anal., \textbf{13}(6), 1239--1280 (2003)

\bibitem{mont}
Montgomery, H. L.:
The pair correlation of zeros of the zeta function.
\newblock Analytic number theory (Proc. Sympos. Pure Math., Vol. XXIV, St. Louis Univ., St. Louis, Mo., 1972), pp. 181–193. Amer. Math. Soc., Providence, R.I., 1973. 

\bibitem{R-S}
Rudnick Z., Sarnak P.:
The pair correlation function of fractional parts of polynomials.
\newblock Comm. Math. Phys. \textbf{194}(1), 61--70 (1998)

\bibitem{rsz}
Rudnick Z., Sarnak P., Zaharescu A.:
The distribution of spacings between the fractional parts of {$n^2\alpha$}.
\newblock Invent. Math., \textbf{145}(1), 37--57 (2001)

\bibitem{RZ}
Rudnick Z., Zaharescu A.:
A metric result on the pair correlation of fractional parts of sequences.
\newblock Acta Arith., \textbf{89}(3), 283--293 (1999)

\bibitem{sarnak} 
Sarnak, P.:
Values at integers of binary quadratic forms.
\newblock Harmonic analysis and number theory (Montreal, PQ, 1996), 181--203, CMS Conf. Proc., 21, Amer. Math. Soc., Providence, RI, 1997. 

\bibitem{Sloan}
Sloan I.:
The method of polarized orbitals for the elastic scattering of slow electrons by ionized helium and atomic hydrogen.
\newblock Proc. Roy. Soc. (London), A 281, 151--163 (1964)

\bibitem{sos}
S\'os V. T.:
On the distribution mod 1 of the sequence $n\alpha$.
\newblock Ann. Univ. Sci. Budap. Rolando E\"otv\"os, Sect. Math. 1, 127--134 (1958)

\bibitem{sp1}
Shparlinski, I.E.: 
On small solutions to quadratic congruences.
\newblock J. Number Theory 131 (2011), no. 6, 1105--1111.

\bibitem{sp2}
Shparlinski, I.E.: 
\newblock On the restricted divisor function in arithmetic progessions. Rev. Mat. Iberoam. 28 (2012), no. 1, 231--238.

\bibitem{Tao}
Tao T., Vu V.:
Additive combinatorics, volume 105 of Cambridge Studies in Advanced Mathematics.
\newblock Cambridge University Press, Cambridge (2006)

\bibitem{td}
Truelsen J. L.:
Divisor problems and the pair correlation for the fractional parts of {$n^2\alpha$}.
\newblock Int. Math. Res. Not. IMRN, (16), 3144--3183 (2010)

\bibitem{van}
Vanderkam, J. M.:
Values at integers of homogeneous polynomials.
\newblock Duke Math. J. 97 (1999), no. 2, 379--412. 

\bibitem{Wein}
Weingartner A.:
The distribution functions of $\sigma (n)/n$ and $n/\varphi \left(n\right)$.
\newblock Proc. Amer. Math. Soc., \textbf{135}, 2677--2681 (2007)


\end{thebibliography}
%

\end{document}